\documentclass[11pt]{article}

\usepackage{amssymb, amsmath, amsthm, graphicx}
\usepackage[left=1in,top=1in,right=1in]{geometry}

\date{}

\theoremstyle{plain}
      \newtheorem{theorem}{Theorem}[section]
      \newtheorem{lemma}[theorem]{Lemma}
            \newtheorem{claim}[theorem]{Claim}
      
      \newtheorem{observation}[theorem]{Observation}

\theoremstyle{definition}
      \newtheorem{definition}[theorem]{Definition}

\theoremstyle{remark}

	\newcommand{\RR}{{\mathbb R}}

\title{Semi-algebraic colorings of complete graphs}
\author{Jacob Fox\thanks{Stanford University, Stanford, CA. Supported by a Packard Fellowship, by NSF CAREER award DMS 1352121. Email: {\tt jacobfox@stanford.edu.}} \and J\'anos Pach\thanks{EPFL, Lausanne and Courant Institute, New York, NY. Supported
by Swiss National Science Foundation Grants 200020-162884 and 200021-175977. Email:
{\tt pach@cims.nyu.edu}.}\and  Andrew Suk\thanks{Department of Mathematics, University of California at San Diego, La Jolla, CA, 92093 USA. Supported an NSF CAREER award and an Alfred Sloan Fellowship. Email: {\tt asuk@ucsd.edu}.} }

\begin{document}

\maketitle

\begin{abstract}

We consider $m$-colorings of the edges of a complete graph, where each color class is defined semi-algebraically with bounded complexity. The case $m = 2$ was first studied by Alon et al., who applied this framework to obtain surprisingly strong Ramsey-type results for intersection graphs of geometric objects and for other graphs arising in computational geometry. Considering larger values of $m$ is relevant, e.g., to problems concerning the number of distinct distances determined by a point set.

For $p\ge 3$ and $m\ge 2$, the classical Ramsey number $R(p;m)$ is the smallest positive integer $n$ such that any $m$-coloring of the edges of $K_n$, the complete graph on $n$ vertices,  contains a monochromatic $K_p$. It is a longstanding open problem that goes back to Schur (1916) to decide whether $R(p;m)=2^{O(m)}$, for a fixed $p$. We prove that this is true if each color class is defined semi-algebraically with bounded complexity.  The order of magnitude of this bound is tight. Our proof is based on the Cutting Lemma of Chazelle {\em et al.}, and on a Szemer\'edi-type regularity lemma for multicolored semi-algebraic graphs, which is of independent interest.
The same technique is used to address the semi-algebraic variant of a more general Ramsey-type problem of Erd\H{o}s and Shelah.


\end{abstract}

\section{Introduction}

The Ramsey number $R(p;m)$ is the smallest integer $n$ such that any $m$-coloring on the edges of the complete $n$-vertex graph contains a monochromatic copy of $K_p$.  The existence of $R(p;m)$ follows from the celebrated theorem of Ramsey \cite{R30} from 1930, and for the special case when $p = 3$, Issai Schur proved the existence of $R(3;m)$ in 1916 in his work related to Fermat's Last Theorem~\cite{schur}. He showed that
        $$\Omega(2^m) \le R(3;m) \le O(m!).$$
While the upper bound has remained unchanged over the last 100 years, the lower bound was successively improved and the current record is $R(3;m) \geq \Omega(3.199^m)$ due to Xiaodong \emph{et al.}~\cite{fred}.  It is a major open problem in Ramsey theory, for which Erd\H os offered some price money, to close the gap between the lower and upper bounds for $R(3;m)$.

In this paper, we study edge-colorings of complete graphs where each color class is defined algebraically with bounded complexity. It is known that several classic theorems in graph theory can be improved for intersection graphs of geometric objects of bounded ``description complexity'' or bounded VC-dimension, graphs of incidences between points and hyperplanes, distance graphs, and, more generally, for semi-algebraic graphs \cite{alon,FPS,CFPSS,FPSSZ,suk}. To make this statement more precise, we need to introduce some terminology.  Let $V$ be an ordered point set in $\RR^d$, and let $E\subset {V\choose 2}$.  We say that $E$ is a \emph{semi-algebraic} relation on $V$ with \emph{complexity} at most $t$ if there are at most $t$ polynomials $g_1,\ldots,g_s\in \RR[x_1,\ldots,x_{2d}]$, $s\leq t$, of degree at most $t$ and a Boolean formula $\Phi$ such that for vertices $u,v \in V$ such that $u$ comes before $v$ in the ordering,

$$(u,v) \in E \hspace{.5cm}\Leftrightarrow\hspace{.5cm} \Phi(g_{1}(u,v) \geq 0;\ldots;g_{s}(u,v) \geq 0) = 1.$$

\noindent  At the evaluation of $g_{\ell}(u,v)$, we substitute the variables $x_1,\ldots,x_d$ with the coordinates of $u$, the variables $x_{d+1},\ldots,x_{2d}$ with the coordinates of $v$.  We may assume that the semi-algebraic relation $E$ is \emph{symmetric}, i.e., for all points $u,v \in \RR^d$, $(u,v) \in E$ if and only if $(v,u) \in E$.  Indeed, given such an ordered point set $V\subset \RR^d$ and a not necessarily symmetric semi-algebraic relation $E$ of complexity at most $t$, we can define $V^{\ast}\subset \RR^{d +1}$ with points $(v,i)$ where $v \in V$ and $v$ is the $i$th smallest element in the given ordering of $V$.  Then we can define a symmetric semi-algebraic relation $E^{\ast}$ on the pairs of $V^{\ast}$ with complexity at most $2t + 2$, by comparing the value of the last coordinates of the two points, and checking the relation $E$ using the first $d$ coordinates of the two points.  We will therefore assume throughout this paper that all semi-algebraic relations we consider are symmetric, and the vertices are not ordered.  Hence, all edges are unordered and we denote $uv = \{u,v\}$.  We also assume that the dimension $d$ and complexity $t$ are fixed parameters, and $n =|V|$ tends to infinity.

Let $R_{d,t}(p;m)$ be the minimum $n$ such that every $n$-element point set $V$ in $\RR^d$ equipped with $m$ semi-algebraic binary relations (edge-colorings) $E_1,\ldots, E_m \subset {V\choose 2}$, each of complexity at most $t$, where $E_1\cup \cdots \cup E_m = {V\choose 2}$, contains a subset $S\subset V$ of size $p$ such that ${S\choose 2} \subset E_k$ for some $k$.  Clearly, $R_{d,t}(p;m) \leq R(p;m)$.  For this setting, it was known that for fixed $d,t \geq 1$, $R_{d,t}(3;m) = 2^{O(m\log\log m)}$, which is much smaller than Schur's bound $R(3;m)=O(m!)$ mentioned in the first paragraph of the Introduction; see~\cite{suk}. In this paper, we completely settle Schur's problem for semi-algebraic graphs, by showing that in this setting Schur's lower bound is tight. In fact, we prove this in a more general form, for any $p\ge 3$.

\begin{theorem}\label{color}
For fixed integers $d,t\geq 1$ and $p\geq 3$, we have

$$R_{d,t}(p;m) = 2^{O(m)}.$$

\end{theorem}

\smallskip Our proof uses geometric techniques and is based on the Cutting Lemma of Chazelle, Edelsbrunner, Guibas, and Sharir~\cite{chazelle} described in Section~\ref{cuttinglemma}.

Edge-colorings of semi-algebraic graphs with $m$ colors can be used, e.g., for studying problems concerning the number of distinct distances determined by a point set; see \cite{distdist}. One can explore the fact that multicolored semi-algebraic graphs have a very nice structural characterization, reminiscent of Szemer\'edi's classic regularity lemma for general graphs \cite{S}, but possessing much stronger homogeneity properties. Our next theorem provides such a characterization, which is of independent interest. To state our result, we need some notation and terminology.

A partition is called {\em equitable} if any two parts differ in size by at most one. According to Szemer\'edi's lemma, for every $\varepsilon>0$ there is a $K=K(\varepsilon)$ such that the vertex set of every graph has an equitable  partition into at most $K$ parts such that all but at most an $\varepsilon$-fraction of the pairs of parts are $\varepsilon$-regular.\footnote{For a pair $(V_i,V_j)$ of vertex subsets, $e(V_i,V_j)$ denotes the number of edges in the graph running between $V_i$ and $V_j$. The density $d(V_i,V_j)$ is defined as $\frac{e(V_i,V_j)}{|V_i||V_j|}$. The pair $(V_i,V_j)$ is called $\varepsilon$-regular if for all $V_i' \subset V_i$ and $V_j' \subset V_j$ with $|V_i'| \geq \varepsilon |V_i|$ and $|V_j'| \geq \varepsilon |V_j|$, we have $|d(V'_i,V'_j)-d(V_i,V_j)| \leq \varepsilon$.} It follows from Szemer\'edi's proof that $K(\varepsilon)$ may be taken to be an exponential tower of $2$s of height $\varepsilon^{-O(1)}$. Gowers \cite{Go97} used a probabilistic construction to show that such an enormous bound is indeed necessary.

Alon~{\em et al.}~\cite{alon} (see also Fox, Gromov~{\em et al.}~\cite{gromov}) established a strengthening of the regularity lemma for point sets in $\RR^d$ equipped with a semi-algebraic relation $E$.  It was shown in \cite{alon} that for any semi-algebraic graph of bounded complexity  defined on the vertex set $V\subset \RR^d$ (that is, for any semi-algebraic binary relation $E\subset {V\choose 2}$), $V$ has an equitable partition into a bounded number of parts such that all but at most an $\varepsilon$-fraction of the pairs of parts $(V_1,V_2)$ behave not only regularly, but {\em homogeneously} in the sense that either $V_1\times V_2\subseteq E$ or $V_1\times V_2\cap E=\emptyset$. The first proof of this theorem was essentially qualitative: it gave a poor estimate for the number of parts in such a partition.  Fox, Pach, and Suk \cite{FPS} gave a stronger quantitative form of this result, showing that the number of parts can be taken to be polynomial in $1/\varepsilon$.

Let $V$ be an $n$-element point set in $\RR^d$ equipped with $m$ semi-algebraic relations $E_1,\ldots, E_m$ such that $E_1\cup \cdots \cup E_m = {V\choose 2}$ of bounded complexity. In other words, suppose that the edges of the complete graph on $V$ are colored with $m$ colors, where each color class is semi-algebraic.  Then, for any $\varepsilon > 0$, an $m$-fold repeated application of the result of Fox, Pach, and Suk \cite{FPS} gives an equitable partition of $V$ into at most $K \leq (1/\varepsilon)^{cm}$ parts such that all but an $\varepsilon$-fraction of the pairs of parts are complete with respect to some relation $E_k$, i.e., all edges between the two parts are of color $k$, for some $k$.  In Section \ref{structure}, we strengthen this result by showing that the number of parts can be taken to be {\em polynomial} in $m/\varepsilon$.

\begin{theorem}
\label{reg1}
For any positive integers $d,t\geq 1$ there exists a constant $c=c(d,t)>0$ with the following property. Let $0<\varepsilon < 1/2$ and let $V$ be an $n$-element point set in $\RR^d$ equipped with semi-algebraic relations $E_1,\ldots, E_m$ such that each $E_k$ has complexity at most $t$ and ${V\choose 2} = E_1\cup \cdots\cup E_m$.  Then $V$ has an equitable partition $V = V_1\cup \cdots \cup V_K$  into at most $4/\varepsilon \leq K\leq (m/\varepsilon)^c$ parts such that all but an $\varepsilon$-fraction of the pairs of parts are complete with respect to some relation $E_k$.
\end{theorem}

In Section~\ref{genram}, we apply this result to solve a problem of Erd\H os and Shelah~\cite{ES81}) in the semi-algebraic setting. Let $d,t,p,q,n$ be positive integers, $p\ge 3$, and $2\le q\le {p\choose 2}$. Let $f_{d,t}(n,p,q)$ be the minimum $m$ such that there exists a semi-algebraic $m$-coloring of the edges of the complete graph of $n$ vertices (with parameters $d$ and $t$, as above) with the property that any $p$ vertices induce at least $q$ distinct colors. Our next theorem precisely determines the smallest $q$ for a given $p$, where  $f_{d,t}(n,p,q)$ changes from $\log n$ to a power of $n$.

\begin{theorem}\label{genramsey}
For fixed integers $d,t\geq 1$, there is a $c = c(d,t) > 0$ such that for $p \geq 3$, we have

$$f_{d,t}(n,p,\lceil\log p\rceil + 1) \geq \Omega\left(n^{\frac{1}{c\log^2p}}\right).$$

\noindent Moreover, for $t\geq 4$,

$$f_{d,t}(n,p, \lceil\log p\rceil) \leq O(\log n).$$

\end{theorem}

In~\cite{distdist}, we studied a geometric instance of this problem, where every set of $p$ points induces at least $q$ {\em distinct distances}.
\smallskip

Our paper is organized as follows. In the next section, we describe the Cutting Lemma of Chazelle {\em et al.}, which is the main geometric tool used in all proofs. In Section~\ref{section3}, we establish Theorem~\ref{color}. Section~\ref{structure} contains the proof of our multicolored semi-algebraic regularity lemma, Theorem~\ref{reg1}, which is then applied in the following section to deduce Theorem~\ref{genramsey}. We end this paper with some concluding remarks.

\section{The cutting lemma}\label{cuttinglemma}

The main tool we use to prove Theorems \ref{color} and \ref{reg1} is commonly referred to as the \emph{cutting lemma}, which we now recall.  A set $\Delta \subset \RR^d$ is \emph{semi-algebraic} if there are polynomials $g_1,\ldots, g_t$ and a boolean formula $\Phi$ such that

$$A = \{x \in \RR^d: \Phi(g_1(x) \geq 0;\ldots; g_t(x) \geq 0) = 1\}.$$

\noindent We say that a semi-algebraic set in $d$-space has \emph{description complexity} at most $t$ if the number of inequalities is at most $t$, and each polynomial $g_i$ has degree at most $t$.  Let $\sigma\subset \RR^d$ be a \emph{surface} in $\RR^d$, that is, $\sigma$ is the zero set of some polynomial $h \in \RR[x_1,\ldots, x_d]$.  The \emph{degree} of a surface $\sigma = \{x \in \RR^d: h(x) = 0\}$ is the degree of the polynomial $h$.  We say that the surface $\sigma\subset \RR^d$ \emph{crosses} a semi-algebraic set $\Delta$ if $\sigma\cap \Delta \neq \emptyset$ and $\Delta \not\subset \sigma$.

Let $\Sigma$ be a collection of surfaces in $\RR^d$, each having bounded degree.  A \emph{$(1/r)$-cutting} for $\Sigma$ is a family $\Psi$ of disjoint (possibly unbounded) semi-algebraic sets of bounded complexity such that

\begin{enumerate}

\item each $\Delta \in \Psi$ is crossed by at most $|\Sigma|/r$ surfaces from $\Sigma$, and

\item the union of all $\Delta \in \Psi$ is $\RR^d$.

\end{enumerate}

\noindent In \cite{chazelle}, Chazelle \emph{et al}.~(see also \cite{koltun}) proved the following.

\begin{lemma}[Cutting lemma]\label{cut}
Let $\Sigma$ be a multiset of $N$ surfaces in $\RR^d$, each surface having degree at most $t$, and let $r$ be an integer parameter such that $1\leq r \leq N$.   Then there is a constant $c_1 = c_1(d,t)$ such that $\Sigma$ admits a $(1/r)$-cutting $\Psi$, where $|\Psi| \leq c_1r^{2d}$, and each semi-algebraic set $\Delta \in \Psi$ has complexity at most $c_1$.

\end{lemma}

\noindent We note that the original statement of Chazelle \emph{et al}.~\cite{chazelle} and Koltun~\cite{koltun} is stronger.  Namely, they also guarantee that the number of cells in the cutting $\Psi$ is at most $r^{2d-4 + \epsilon}$ for $d\geq 4$.  Here, for simplicity, we use the weaker bound of $c_1r^{2d}$, as stated above.

\section{Multicolor Ramsey numbers for small cliques\\--Proof of Theorem~\ref{color}}\label{section3}

Theorem \ref{color} will easily follow from Theorem \ref{secmulti} below.  For integers $p_1,\ldots, p_m \geq 2$, $d,t \geq 1$, let $R_{d,t}(p_1,\ldots, p_m)$ be the minimum integer $n$ with the following property.   Every complete graph $K_n$, whose $n$ vertices lie in $\RR^d$ and whose edges are colored with $m$ colors such that each color class is defined by a semi-algebraic relation of description complexity $t$, contains a monochromatic copy of $K_{p_k}$ in color $k$ for some $1\leq k \leq m$.

\begin{theorem}\label{secmulti}

For any $d,t\ge 1$ and $p\ge 3$, there exists a constant $c=c(d,t,p)$ satisfying the following condition. For any $m$ integers $p_1,\ldots,p_m\le p$, we have

$$R_{d,t}(p_1,\ldots, p_m) \leq 2^{c\sum_{k =1 }^m p_k}.$$

\end{theorem}

\begin{proof}  Fix $d,t\geq 1$, $p \geq 3$ and set $c= c(d,t,p)$ to be a large constant that will be determined later.  We will show that $R_{d,t}(p_1,\ldots, p_m) \leq 2^{c \sum_{k=1}^m p_k}$ by induction on $s = \sum_{k=1}^m p_k$.  The base case $s \leq 10\cdot 2^{10dtp}$ follows for $c$ sufficiently large.

Now assume that the statement holds for $s' < s$.  Set $n = 2^{cs}$ and let $V$ be an $n$-element point set in $\RR^d$ equipped with semi-algebraic relations $E_1,\ldots,E_m\subset {V\choose 2}$ such that ${V\choose 2} = E_1\cup \cdots \cup E_m$ and each $E_k$ has complexity at most $t$.  Let us remark the an edge $uv$ may have several colors.  We will show that there is a subset $S\subset V$ of size $p_k$ such that ${S\choose 2} \subset E_k$ for some $k \in [m]$, in other words, we will find a monochromatic copy of $K_{p_k}$ in color $k$ for some $k \in [m]$.   Throughout the proof, we will let $c_1$ be as defined in Lemma \ref{cut}.

For each relation $E_k$, there are $t$ polynomials $g_{k,1},\ldots,g_{k,t}$ of degree at most $t$, and a Boolean function $\Phi_k$ such that

$$uv \in E_k \hspace{.5cm}\Leftrightarrow\hspace{.5cm} \Phi_k(g_{k,1}(u,v) \geq 0,\ldots,g_{k,t}(u,v)\geq 0) = 1.$$

\noindent For $ 1 \leq k \leq m,  1\leq \ell \leq t, v \in V$, we define the surface $\sigma_{k,\ell}(v) = \{ x \in \RR^d: g_{k,\ell}(v,x)  = 0\}$.

Before we continue, let us briefly sketch the idea of the proof.  We start by applying Lemma~\ref{cut} (the cutting lemma) to  $\Sigma = \{\sigma_{k,\ell}(v): k\in[m], \ell \in [t],v \in V\}$  and obtain a space partition which induces a partition of the vertex set $V  = V_1\cup \cdots \cup V_K$.  If there is a ``large" part $V_j$ with many distinct colors appearing in $V_j\times (V\setminus V_j)$, then we show that $V_j$ induces few distinct colors, and by induction we can find a monochromatic copy of $K_{p_k}$ for some $k\in [m]$.  If none of the ``large" parts has the above property, the colors of nearly all edges can be defined by much fewer polynomial inequalities, i.e., by a much small set of surfaces $\Sigma'\subset \Sigma$ .  Now we can repeat.

In what follows, we spell out these ideas in full detail.   Set $m_0 = m$ and define $m_i = 4d\log(c_1m_{i-1}t)$ for $i > 0$.  We will establish the following claim.

\begin{claim}
Let $V$ and $E_1,\ldots, E_m \subset {V\choose 2}$ be defined as above.  Then we will recursively find either

\begin{enumerate}

\item a monochromatic copy of $K_{p_k}$ in color $k$ for some $k \in [m]$, or

\item a function $\chi_i:V\rightarrow 2^{[m]}$ such that $|\chi_i(v)| \leq m_{i}$, and the number of edges $uv \in {V\choose 2}$ with the property that for one of its endpoints, say $u$, no color assigned to $uv$ belongs to $\chi_i(u)$, is at most $\frac{4n^2}{tm_{i-1}}$.  We will refer to these edges as bad at stage $i$.  All edges that are not bad are called good at this stage, meaning that, there is a color $k$ appearing on $uv$ such that $k \in \chi_i(u)$, and there is a color $k'$ appearing on $uv$ such that $k' \in \chi_i(v)$.

\end{enumerate}

\end{claim}

\begin{proof}  We start by setting $\chi_0(v) = [m]$ for all $v \in V$, and $m_0 = m$.  Having found $\chi_i$ with the properties above, we will produce $\chi_{i+1}$ as follows.  We have $m_{i + 1} = 4d\log(c_1m_{i}t)$, and let us assume that $m_{i} > (8c_1dtp)^2$.  Hence, there are at most $ \frac{4n^2}{tm_{i-1}}$ bad edges.   Let $\Sigma$ be the set of surfaces $\sigma_{k,\ell}(v)$, where $v \in V$, $k \in \chi_i(v)$, and $1 \leq \ell \leq t$.  This implies that $|\Sigma| \leq nm_it$.

We apply Lemma \ref{cut} to $\Sigma$ with parameter $r = (tm_i)^2$ to obtain a $(1/(tm_i)^2)$-cutting $\Psi  = \{\Delta_{1}, \Delta_{2},\ldots, \Delta_{K_0}\}$, such that $K_0 \leq c_1(tm_i)^{4d}$.  Hence, we have a partition $\mathcal{P}_0 : V = V_1\cup \cdots \cup V_{K_0}$, where $V_j = V\cap \Delta_j$ for $\Delta_{j} \in \Psi$.   For each part $V_j$ of size greater than $2n/(tm_i)$, we (arbitrarily) partition $V_j$ into parts of size $\lfloor 2n/(tm_i)\rfloor$ and possibly one additional part of size less than $2n/(tm_i)$.  Let $\mathcal{P}: V = V_1\cup \cdots \cup V_{K}$ be the resulting partition, where $K \leq 2c_1(tm_i)^{4d}$ and $|V_j| \leq 2n/(tm_i)$ for all $j$.

\medskip

Now we define $\chi_{i + 1}(v)$ for all $v \in V$.

\medskip

\noindent \emph{Case 1}.  If $v \in V_j$ for some $V_j$ with $|V_j|< \frac{n}{2c_1(tm_i)^{4d + 1}}$, we set $\chi_{i + 1}(v) = \emptyset$.

\medskip

\noindent \emph{Case 2}.  Suppose $v \in V_j$ such that $|V_j| \geq \frac{n}{2c_1(tm_i)^{4d + 1}}$.  In order to define $\chi_{i + 1}(v)$, we need some preparation.  Let $\Delta_j \in \Psi$ such that $V_j \subset \Delta_j$.  We define $X_j \subset V\setminus V_j$ to be the set of vertices from $V\setminus V_j$ that gives rise to a surface in $\Sigma$ that crosses $\Delta_j$.  Hence $|X_j| \leq n/(tm_i)$. Fix a vertex $v \in V\setminus \{V_j,X_j\}$.  Since none of the surfaces of the form $\sigma_{k,\ell}(v)$, where $k \in \chi_i(v)$ and $\ell \in \{1,\ldots, t\}$, cross $\Delta_j$, either $v\times V_j$ is monochromatic with color $k$ for some $k \in \chi_i(v)$, or none of the colors in $\chi_i(v)$ appear in $v\times V_j$.  Let $S_j$ be the set of vertices $v \in V\setminus \{V_j,X_j\}$ satisfying the former condition and let $T_j$ denote the set of vertices $v \in V\setminus \{V_j,X_j\}$ satisfying the latter one.  Since there are at most $4n^2/(tm_{i-1})$ bad edges, we have

$$|T_j| \frac{n}{2c_1(tm_i)^{4d + 1}} \leq \frac{4n^2}{tm_{i-1}},$$

\noindent which implies

$$|T_j| \leq \frac{8nc_1(tm_i)^{4d+1}}{tm_{i-1}} \leq \frac{n}{tm_i},$$

\noindent where the last inequality follows from the assumption $m_i > (8c_2dtp)^2$.  Now, suppose there are at least $m_{i + 1} = 4d\log(c_1tm_i)$ distinct colors between $V_j$ and $S_j$.  Let $I = \{k_1,\ldots, k_{m_{i+1}}\} \subset [m]$ be the set of these $m_{i + 1}$ distinct colors.  Then there are $m_{i + 1}$ vertices $v_1,\ldots, v_{m_{i  + 1}} \in S_j$, possibly with repetition, such that $v_w\times V_j$ is monochromatic with color $k_w \in I$, for each $w \in \{1,\ldots, m_{i + 1}\}$.  Hence, if $V_j$ contains a monochromatic copy of $K_{p_k -1}$ in color $k \in I$, we would have a monochromatic copy of $K_{p_k}$ in color $k$.  On the other hand, if $V_j$ does not contain a monochromatic copy of $K_{p_k - 1}$ in color $k$ for no $k \in I$, then, using that

 $$|V_j| \geq \frac{n}{2c_1(tm_i)^{4d+1}} > 2^{cs - 8d\log(c_1m_i t)} > 2^{c( s- m_{i + 1})} = 2^{c\left(\sum_{k \in I}(p_k - 1) + \sum_{k\not\in I} p_k\right)} $$

\noindent for a sufficiently large $c$, we obtain by induction that there is a monochromatic copy of $K_{p_k}$ in color $k$ where $k \not\in I$.

Therefore, we can assume that the number of distinct colors between $V_j$ and $S_j$ is less than $m_{i + 1} = 4d\log(c_1 m_{i}t)$.  For every vertex $v \in V_j$, define $\chi_{i + 1}(v)$ as the set of all colors that appear on the edges belonging to $v\times S_j$.

Now that we have defined $m_{i + 1}$ and $\chi_{i+1}$ such that $|\chi_{i + 1}(v)| \leq m_{i+1}$ for all $v\in V$, it remains to show that the number of edges $uv \in {V\choose 2}$ with the property that for one of its endpoints, say $u$, no color assigned to $uv$ belongs to $\chi(u)$, is at most $\frac{4n^2}{tm_i}$.  Let $B\subset {V\choose 2}$ be the collection of such edges. Notice that if $uv \in B$, then either

\noindent

\begin{enumerate}

\item both $u$ and $v$ lie inside the same part in the partition $\mathcal{P}$, or

\item $u$ or $v$ lies inside a part $V_j$ such that $|V_j| < \frac{n}{2c_1(tm_i)^{4d + 1}}$, or

\item $u \in V_j$ with $|V_j|\geq \frac{n}{2c_1(tm_i)^{4d + 1}}$ and $v \in X_j\cup T_j$, or

\item $v \in V_j$ with $|V_j|\geq \frac{n}{2c_1(tm_i)^{4d + 1}}$ and $u \in X_j\cup T_j$.

\end{enumerate}

\noindent The number of edges of type 1 is at most $\frac{tm_i}{2}\left(\frac{2n}{tm_i}\right)^2/2 = n^2/(tm_i)$.  The number of edges of type~2 is at most

$$\left(\frac{n}{2c_1(tm_i)^{4d+1}}\right)K\cdot n \leq \frac{n^2}{tm_i}.$$

\noindent Since $|X_j|,|T_j| \leq n/(tm_i)$, the number of edges of types 3 and 4 is at most

$$\sum_j |V_j|\frac{2n}{tm_i} \leq \frac{2n^2}{tm_i}.$$

\noindent Hence, $|B| \leq \frac{4n^2}{tm_i}$.  Therefore, either we have found a monochromatic copy of $K_{p_k}$ in color $k$ for some $k \in [m]$, or we have found $m_{i + 1}$ and $\chi_{i + 1}$ with the desired properties.\end{proof}

Let $w$ be the minimum integer such that $m_{w} \leq (8c_1dtp)^2$. Then either we have found a monochromatic copy of $K_{p_{k}}$ in color $k$ for some $k\in [m]$, or we have obtained $m_{w}$ and $\chi_w$ with the desired properties.  Since there are at most $4n^2/(tm_{w-1}) < n^2/8$ bad edges, there is a vertex $v\in V$ incident to at least $n/2$ good edges.  Moreover, since $\chi_w(v) \leq m_w \leq (8c_1dtp)^2$, at least $\frac{n}{2(8c_1dtp)^2}$ of these edges incident to $v$ have color $k'$ for some color $k' \in \chi_w(v)$.  Let $S\subset V$ be the set of endpoints of these edges.  If $S$ contains a monochromatic copy of $K_{p_{k'}-1}$ in color $k'$, then we are done.  On the other hand, if $S$ does not contain a monochromatic copy of $K_{p_{k'} - 1}$ in color $k'$, and using the lower bound

$$|S|\geq \frac{n}{2(8c_1dtp)^2} = \frac{2^{cs}}{2(8c_1dtp)^2} \geq 2^{c(\sum_{k \neq k'}p_k + (p_{k'}-1))},$$

\noindent for $c = c(d,t,p)$ sufficiently large, we conclude by induction that $S$ contains a monochromatic copy of $K_{p_{k}}$ for some $k\neq k'$.  This completes the proof of Theorem \ref{secmulti}.\end{proof}

\section{Multicolor semi-algebraic regularity lemma\\--Proof of Theorem~\ref{reg1}}\label{structure}

First, we prove the following variant of Theorem \ref{reg1}, which easily implies Theorem \ref{reg1}.

\begin{theorem}
\label{reg}
For any $\varepsilon > 0$, every $n$-element point set $V\subset \RR^d$ equipped with semi-algebraic binary relations $E_1,\ldots, E_m \subset {V\choose 2}$ such that ${V\choose 2} = E_1\cup \cdots \cup E_m$ and each $E_k$ has complexity at most $t$, can be partitioned into $K\leq c_2(\frac{m}{\varepsilon})^{5d^2}$ parts $V = V_1\cup \cdots \cup V_K$, where $c_2 = c_2(d,t)$, such that

$$\sum\frac{|V_{i}||V_{j}|}{n^2} \leq \varepsilon,$$

\noindent where the sum is taken over all pairs $(i,j)$ such that $(V_{i},V_{j})$ is not complete with respect to $E_{k}$ for all $k = 1,\ldots, m$.
\end{theorem}

\begin{proof}  For each relation $E_k$, let $g_{k,1},\ldots, g_{k,t} \in \mathbb{R}[x_1,\ldots, x_{2d}]$ be polynomials of degree at most $t$, and let $\Phi_k$ be a boolean formula such that

$$uv \in E_k \hspace{.5cm}\Leftrightarrow\hspace{.5cm} \Phi_k(g_{k,1}(u,v) \geq 0;\ldots;g_{k,t}(u,v) \geq 0) = 1.$$

For each point $x\in \RR^d$, $k\in \{1,\ldots, m\}$, and $\ell \in \{1,\ldots, t\}$, we define the surface $$\sigma_{k,\ell}(x) = \{y \in \RR^d: g_{k,\ell}(x,y) = 0\}.$$  Let $\Sigma$ be the family of $tmn$ surfaces in $\RR^d$ defined by

 $$\Sigma = \{\sigma_{k,\ell}(u): u \in V, 1 \leq  k\leq m, 1\leq \ell\leq t\}.$$

 We apply Lemma \ref{cut} to $\Sigma$ with parameter $r = tm/\varepsilon$ to obtain a $(1/r)$-cutting $\Psi$, where $|\Psi| = s \leq c_1\left(\frac{tm}{\varepsilon}\right)^{2d}$, such that each semi-algebraic set $\Delta_i \in \Psi$ has has complexity at most $c_1$, where $c_1$ is defined in Lemma~\ref{cut}.  Hence, at most $tmn/r  = \varepsilon n$ surfaces from $\Sigma$ cross $\Delta_{i}$ for every $i$.  This implies that at most $\varepsilon n$ points in $V$ give rise to at least one surface in $\Sigma$ that cross $\Delta_{i}$.

Let $U_{i} = V\cap \Delta_{i}$ for each $i \leq s$.  We now partition $\Delta_{i}$ as follows. For $k \in \{1,\ldots , m\}$ and $j \in \{1,\ldots, s\}$, define $\Delta_{i,j,k} \subset \RR^d$ by

 $$\Delta_{i,j,k} = \{x \in \Delta_{i}: \sigma_{k,1}(x) \cup \cdots \cup \sigma_{k,t}(x) \textnormal{ crosses } \Delta_j\}.$$

 \begin{observation}
For any $i$, $j$, and $k$, the semi-algebraic set $\Delta_{i,j,k}$ has complexity at most $c_3 = c_3(d,t)$.

 \end{observation}

 \begin{proof}

Set $\sigma_k(x) =  \sigma_{k,1}(x) \cup \cdots \cup \sigma_{k,t}(x)$, which is a semi-algebraic set with complexity at most $c_4 = c_4(d,t)$.  Then

$$\Delta_{i,j,k} = \left\{x \in \Delta_{i}: \begin{array}{l}
                                            \exists y_1 \in \RR^d \textnormal{ s.t. } y_1\in \sigma_k(x)\cap \Delta_{j}, \textnormal{ and } \\
                                             \exists y_2 \in \RR^d \textnormal{ s.t. } y_2\in   \Delta_{j}\setminus \sigma_k(x).
                                          \end{array}\right\}.$$

We can apply quantifier elimination (see Theorem 2.74 in \cite{basu}) to make $\Delta_{i,j,k}$ quantifier-free, with description complexity at most $c_3 = c_3(d,t)$. \end{proof}

Set $\mathcal{F}_{i} = \left\{\Delta_{i,j,k}: 1 \leq k \leq m, 1 \leq j \leq s \right\}$.  We partition the points in $U_{i}$ into equivalence classes, where two points $u,v \in U_{i}$ are equivalent if and only if $u$ belongs to the same members of $\mathcal{F}_{i}$ as $v$ does.  Since $\mathcal{F}_{i}$ gives rise to at most $c_3|\mathcal{F}_{i}|$ polynomials of degree at most $c_3$, by the Milnor-Thom theorem (see \cite{mat} Chapter 6), the number of distinct sign patterns of these $c_3|\mathcal{F}_{i}|$ polynomials is at most $\left(50c_3(c_3|\mathcal{F}_{i}|)\right)^d.$  Hence, there is a constant $c_5 = c_5(d,t)$ such that $U_{i}$ is partitioned into at most $c_5(ms)^d$ equivalence classes.  After repeating this procedure to each $U_{i}$, we obtain a partition of our point set $V =  V_1\cup \cdots \cup V_K$ with

$$K \leq sc_5(ms)^d =c_5m^ds^{d+1} \leq c_5 t^{2d(d+1)}c_1^{d+1} \left(\frac{m}{\varepsilon}\right)^{5d^2} = c_2 \left(\frac{m}{\varepsilon}\right)^{5d^2},$$
where we define $c_2 = c_5 t^{2d(d+1)}c_1^{d+1}$.

For fixed $i$, consider the part $V_i$.  Then there is a semi-algebraic set $\Delta_{w_i}$ obtained from Lemma~\ref{cut} such that $U_{w_i} = V\cap \Delta_{w_i}$ and  $V_i \subset U_{w_i} \subset \Delta_{w_i}$.  Now consider all other parts $V_j$ such that not all of their elements are related to every element of $V_i$ with respect to any relation $E_k$ where $1\leq k \leq m$.  Then each point $u\in V_j$ gives rise to a surface in $\Sigma$ that crosses $\Delta_{w_i}$.  By Lemma \ref{cut}, the total number of such points in $V$ is at most $\varepsilon n$.  Therefore, we have

$$\sum\limits_j |V_i||V_j| = |V_i| \sum\limits_j |V_j| \leq |V_i|\varepsilon n,$$

\noindent where the sum is over all $j$ such that $V_i\times V_j$ is not contained in the relation $E_{k}$ for any $k$.  Summing over all $i$, we have

$$\sum\limits_{i,j} |V_{i}||V_{j}| \leq \varepsilon n^2,$$

\noindent where the sum is taken over all pairs $i,j$ such that $(V_{i},V_{j})$ is not complete with respect to $E_{k}$ for all $k$.\end{proof}

\begin{proof}[Proof of Theorem \ref{reg1}.] Apply Theorem \ref{reg} with approximation parameter $\varepsilon/2$. Hence, there is a partition $\mathcal{Q}:V=U_1 \cup \cdots \cup U_{K'}$ into $K' \leq (m/\varepsilon)^c$ parts with $c=c(d,t)$ and $\sum |U_{i}||U_{j}|\leq (\varepsilon/2) |V|^2$, where the sum is taken over all pairs $(i,j)$ such that $(U_{i},U_{j})$ is not complete with respect to $E_{k}$ for all $k$.

Let $K=8\varepsilon^{-1}K'$. Partition each part $U_i$ into parts of size $|V|/K$ and possibly one additional part of size less than $|V|/K$. Collect these additional parts and divide them into parts of size $|V|/K$ to obtain an equitable partition $\mathcal{P}:V=V_1 \cup \cdots \cup V_K$ into $K$ parts. The number of vertices of $V$ which are in parts $V_i$ that are not contained in a part of $\mathcal{Q}$ is at most $K'|V|/K$. Hence, the fraction of pairs $V_{i} \times V_{j}$ with not all $V_{i},V_{j}$ are subsets of parts of $\mathcal{Q}$ is at most $2K'/K=\varepsilon/4$. As $\varepsilon/2+\varepsilon/4<\varepsilon$, we obtain that less than an $\varepsilon$-fraction of the pairs of parts of $\mathcal{P}$ are not complete with respect to any relation $E_1,\ldots, E_m$.\end{proof}

\section{Generalized Ramsey numbers for semi-algebraic colorings\\
--Proof of Theorem~\ref{genramsey}}\label{genram}

Due to the lack of understanding of the classical Ramsey number $R(p;m)$, Erd\H os and Shelah (see \cite{ES81}) introduced the following generalization, which was studied by Erd\H os and Gy\'arf\'as in \cite{EG}.

\begin{definition}
For integers $p$ and $q$ with $2 \leq q \leq {p\choose 2}$, a $(p,q)$\emph{-coloring} is an edge-coloring of a complete graph in which every $p$ vertices induce at least $q$ distinct colors.
\end{definition}

\noindent Let $f(n,p,q)$ be the minimum integer $m$ such that there is a $(p,q)$-coloring of $K_n$ with at most $m$ colors.  Here, both $p$ and $q$ are considered fixed integers, where $p\geq 3$, $2 \leq q \leq {p\choose 2}$, and $n$ tends to infinity.  Trivially, we have $f(n,p,{p\choose 2}) = {n\choose 2}$, and at the other end, estimating $f(n,p,2)$ is equivalent to estimating $R(p;m)$ since $f(n,p,2)$ is the inverse of $R(p;m)$.  In particular,

\begin{equation}\label{sbound}\Omega\left(\frac{\log n}{\log\log n}\right) \leq f(n,3,2) \leq O(\log n).\end{equation}

\noindent  Erd\H os and Gy\'arf\'as \cite{EG} determined certain ranges for $q \in \{2,3,\ldots, {p\choose 2}\}$ for which $f(n,p,q)$ is quadratic, linear, and subpolynomial in $n$.  In particular, they showed that

$$\Omega\left(n^{\frac{1}{p-2}}\right) \leq f(n,p,p) \leq O\left( n^{\frac{2}{p-1}} \right),$$

\noindent which implies that $f(n,p,q)$ is polynomial in $n$ for $q \geq p$.  Surprisingly, estimating $f(n,p,p-1)$ is much more difficult.  They \cite{EG} asked for $p$ fixed if $f(n,p,p-1)=n^{o(1)}$ . The trivial lower bound is $f(n,p,p-1) \geq f(n,p,2) \geq \Omega\left(\frac{\log n}{\log \log n}\right)$, which was improved by several authors \cite{KM,FS}, and it is now known \cite{CFLS} that $f(n,p,p-1) \geq \Omega(\log n)$.  In the other direction, Mubayi \cite{mubayi} found an elegant construction which implies $f(n,4,3) \leq e^{O(\sqrt{\log n})}$, and later, Conlon \emph{et al.}~\cite{CFLS} gave another example which implies $f(n,p,p-1) \leq e^{(\log n)^{1 - 1/(p-2)+o(1)}}$. Hence, it is now known that $f(n,p,p-1)$ does not grow as a power in $n$.

Here, we study the variant of the function $f(n,p,q)$ for point sets $V\subset \RR^d$ equipped with semi-algebraic relations.  Let $f_{d,t}(n,p,q)$ be the minimum $m$ such that there is a $(p,q)$-coloring of $K_n$ with $m$ colors, whose vertices can be chosen as points in $\RR^d$, and each color class can defined by a semi-algebraic relation on the point set with complexity at most $t$.  We note that here we require that each edge receives \emph{exactly} one color. Clearly, we have $f(n,p,q) \leq f_{d,t}(n,p,q)$.  Theorem~\ref{genramsey} stated in the Introduction shows the exact value of $q$ for which $f_{d,t}(n,p,q)$ changes from $\log n$ to a power of $n$.
\medskip

In the rest of this section, we prove Theorem \ref{genramsey}.  Let $V$ be a set of points in $\RR^d$ equipped with semi-algebraic relations $E_1,\ldots, E_m$ such that each $E_k$ has complexity at most $t$, ${V\choose 2} = E_1\cup \cdots\cup E_m$, and $E_k\cap E_{\ell} = \emptyset$ for all $k\neq \ell$.  Let $S_1,S_2 \subset V$ be $q$-element subsets of $V$.  We say that $S_1$ and $S_2$ are \emph{isomorphic}, denoted by $S_1 \simeq S_2$, if there is a bijective function $h:S_1 \rightarrow S_2$ such that for $u,v \in S_1$ we have $uv \in E_k$ if and only if $h(u)h(v) \in E_k$.

Let $S\subset V$ be such that $|S| = 2^s$ for some positive integer $s$.  We say that $S$ is $s$-\emph{layered} if $s = 1$ or if there is a partition $S = S_1\cup S_2$ such that $|S_1| = |S_2| = 2^{s-1}$, $S_1$ and $S_2$ are $(s-1)$-layered, $S_1\simeq S_2$, and for all $u \in S_1$ and $v\in S_2$ we have $uv \in E_k$ for some fixed $k$.  Notice that given an $s$-layered set $S$, there are at most $s$ relations $E_{k_1},\ldots, E_{k_s}$ such that ${S\choose 2} \subset E_{k_1}\cup\cdots \cup E_{k_s}$.   Hence, the lower bound in Theorem \ref{genramsey} is a direct consequence of the following result.

\begin{theorem}
Let $s\geq 1$ and let $V$ be an $n$-element point set in $\RR^d$ equipped with semi-algebraic relations $E_1,\ldots, E_m$ such that each $E_k$ has complexity at most $t$,  $E_1\cup\cdots \cup E_m = {V\choose 2}$, and $E_k\cap E_{\ell} = \emptyset$ for all $k\neq \ell$.  If $m \leq n^{\frac{1}{cs^2}}$, then there is a subset $S\subset V$ such that $|S| = 2^s$ and $S$ is $s$-layered, where $c = c(d,t)$.
\end{theorem}

\begin{proof}

We proceed by induction on $s$.  The base case $s = 1$ is trivial.  For the inductive step, assume that the statement holds for $s' < s$.  We will specify $c = c(d,t)$ later. We start by applying Theorem \ref{reg1} with parameter $\varepsilon = \frac{1}{m^{s}}$ to the point set $V$, which is equipped with semi-algebraic relations $E_1,\ldots, E_m$, and obtain an equitable partition $\mathcal{P}:V = V_1\cup\cdots\cup V_K$, where

$$K \leq c_2\left(\frac{m}{\varepsilon}\right)^{5d^2} \leq c_2 m^{10sd^2}, $$

\noindent and $c_2 = c_2(d,t)$.  Since all but an $\varepsilon$ fraction of the pairs of parts in $\mathcal{P}$ are complete with respect to $E_{k}$ for some $k$, by Tur\'an's theorem, there are $m^{s-1} +1$ parts $V'_i \in \mathcal{P}$ such that each pair $(V'_i,V'_j) \in \mathcal{P} \times \mathcal{P}$ is complete with respect to some relation $E_{k}$.    Since $\mathcal{P}$ is an equitable partition, we have $|V'_i| \geq \frac{n}{c_2m^{10d^2s}}$.  By picking $c = c(d,t)$ sufficiently large, we have

 $$|V_i'|^{\frac{1}{c(s-1)^2}} \geq \left(\frac{n}{c_2m^{10d^2s}}\right)^{\frac{1}{c(s-1)^2}} \geq  m^{\frac{cs^2 - 10c_2d^2s}{c(s-1)^2}} \geq m.$$

\noindent  By the induction hypothesis, each $V'_i$ contains an $(s-1)$-layered set $S_i$ for $i \in \{1,\ldots, m^{s-1} + 1\}$.  By the pigeonhole principle, there are two $(s-1)$-layered sets $S_i,S_j$ such that $S_i \simeq S_j$.  Since $S_i\times S_j \subset E_{k}$ for some $k$, the set $S = S_i\cup S_j$ is an $s$-layered set.  This completes the proof.\end{proof}

To prove the upper bound for $f_{d,t}(n,p,\lceil\log p\rceil )$, when $d\geq 1$ and $t\geq 100$, it is sufficient to construct a $2^m$-element point set $V\subset \RR$ equipped with semi-algebraic relations $E_1,\ldots, E_m$ that is $m$-layered.  More precisely, for each integer $m \geq 1$, we construct a set $V_m$ of $2^m$ points in $\RR$ equipped with semi-algebraic relations $E_1,\ldots, E_m$ such that

\begin{enumerate}

\item $V_m$ with respect to relations $E_1,\ldots, E_m$ is $m$-layered,

\item $E_1\cup \cdots \cup E_m = {V_m\choose 2}$ is a partition,

\item each $E_i$ has complexity at most four, and

\item each $E_i$ is \emph{shift invariant}, that is $uv \in E_i$ if and only if $(u + c, y + c) \in E_i$ for $c \in \RR$.

\end{enumerate}

We start by setting $V_1 = \{1,2\}$ and defining $E_1 = \{u,v \in V_1: |u - v| = 1\}$.  Having defined the point set $V_i$ and relations $E_1,\ldots, E_i$, we define $V_{i + 1}$ and $E_{i + 1}$ as follows.  Let $C = C(i)$ be a sufficiently large integer such that $C > 10\max_{u \in V_i} u$.  Then we have $V_{i + 1} = V_i \cup (V_i + C)$, where $V_i + C$ is a translated copy of $V_i$.  We now define the relation $E_{i + 1}$ by

$$uv \in E_{i + 1} \hspace{.5cm}\Leftrightarrow \hspace{.5cm} C/2 < |u-v| < 2C.$$

Hence, $V_{i+1}$ with respect to relations $E_1,\ldots, E_{i + 1}$ satisfies the properties stated above and is clearly $(i + 1)$-layered.  One can easily check that any set of $p$ points in $V_m$ induces at least $\lceil \log p \rceil$ distinct relations (colors).

\medskip

Let us remark that the arguments above hold for semi-algebraic relations $E_1,\ldots, E_m$ that are not necessarily disjoint if one defines a  $(p,q)$-coloring as follows.  Given a coloring $\chi:{V(K_n)\choose 2}\rightarrow 2^{[m]}$ on the edges of $K_n$, where each edge receives \emph{at least} one color among $[m]$, $\chi$ is a \emph{$(p,q)$-coloring} if for every set $S\subset V$ of size $p$, no matter how you choose one color in $\chi(uv)$ for each edge $uv \in {S\choose 2}$, $S$ will induce at least $q$ distinct colors.

\section{Concluding remarks}

In \cite{suk}, it was shown that $R_{1,t}(3;m) > (1681)^{m/7}$ for $t > 5$, thus implying that the upper bound in Theorem \ref{color} is tight up to a constant factor in the exponent.  This can be improved as follows.  Let $C(p) = \lim_{m \to \infty} R(p;m)^{1/m}$. Note that this limit exists by considering product colorings, but may be finite or infinite.  Then for each $C < C(p)$, there is a $t = t(C,p)$, such that for all $m$ sufficiently large we have

$$R_{1,t}(p;m) > C^m.$$

\noindent Indeed, take a fixed coloring of the edges of $K_N$ which realizes $R(p;m_0)>C^{m_0}$, and recursively blow up this graph by introducing $m_0$ new colors at each stage.  Then this coloring can be realized semi-algebraically in $\mathbb{R}$ with $t=O(m_0^2)$ linear constraints for each color class based on distances.

\end{document}